\documentclass[11pt]{article}
\usepackage{amsmath, amssymb, amsfonts, amsthm, graphicx, relsize, mathtools, mathrsfs, euscript, bbm, bbold, slashbox, float, hyperref, enumitem, url, breakurl}
\usepackage[labelsep=period, font=footnotesize, labelfont=bf]{caption}
\usepackage[T1]{fontenc}
\usepackage[title]{appendix}
\usepackage{titlesec}
\titlelabel{\thetitle.\,\,}
\usepackage{xcolor}
\hypersetup{
    colorlinks,
    linkcolor={black!50!black},
    citecolor={black!50!black},
    urlcolor={black!80!black}
}
\setlist[itemize]{noitemsep,topsep=0pt}
\allowdisplaybreaks

\theoremstyle{plain}
\newtheorem{theorem}{Theorem}[section]
\newtheorem{lemma}[theorem]{Lemma}

\newtheorem{proposition}[theorem]{Proposition}

\theoremstyle{definition}

\newtheorem{remark}[theorem]{Remark}

\DeclareMathAlphabet{\mathbbmsl}{U}{bbm}{m}{sl}

\DeclareMathAlphabet{\mathpzc}{OT1}{pzc}{m}{it}

\def\urltilda{\kern-0.15em\lower0.7ex\hbox{\~{}}\kern0.04em}
\newcommand{\FF}{\mathbbmsl{F}}
\DeclareMathOperator{\rank}{rank}

\begin{document}

\title{\bf{On the classification  of skew Hadamard matrices of order $\boldsymbol{36}$ and related structures}\\[9mm]}

\author{Makoto Araya$^{^1}$  \qquad Masaaki Harada$^{^2}$  \qquad   Hadi Kharaghani$^{^3}$  \\      Ali Mohammadian$^{^4}$  \qquad  Behruz Tayfeh-Rezaie$^{^5}$  \\[4mm]
$^{^1}$Department of Computer Science, Shizuoka University,\\ Hamamatsu 432-8011, Japan\\
$^{^2}$Research Center for Pure and Applied Mathematics, \\ Graduate School of Information Sciences,   Tohoku University, \\ Sendai 980-8579, Japan \\
$^{^3}$Department of Mathematics and Computer Science, University of Lethbridge, \\ Lethbridge, Alberta, T1K 3M4, Canada \\
$^{^4}$School of Mathematical Sciences,    Anhui University, \\ Hefei 230601,  Anhui,  China \\
$^{^5}$School of Mathematics, Institute for Research in Fundamental Sciences (IPM), \\ P.O. Box 19395-5746, Tehran, Iran \\[4mm]
\href{mailto:araya@inf.shizuoka.ac.jp}{araya@inf.shizuoka.ac.jp} \qquad
\href{mailto:mharada@tohoku.ac.jp}{mharada@tohoku.ac.jp} \qquad
\href{mailto:kharaghani@uleth.ca}{kharaghani@uleth.ca} \\
\href{mailto:ali\_m@ahu.edu.cn}{ali\_m@ahu.edu.cn}  \qquad
\href{mailto:tayfeh-r@ipm.ir}{tayfeh-r@ipm.ir}\vspace{9mm}}

\date{}

\sloppy

\maketitle

\begin{abstract}
\noindent
Two skew Hadamard matrices are considered {\sf SH}-equivalent if they are similar by a signed permutation matrix.  This paper determines the number of    {\sf SH}-inequivalent skew Hadamard matrices of order  $36$ for some types.
We also study ternary self-dual codes and association schemes constructed from the skew Hadamard matrices of order $36$.  \\[3mm]
\end{abstract}

\section{Introduction}

A {\sl  Hadamard  matrix}  of order $n$ is an $n\times n$  matrix $H$ with entries in $\{-1, 1\}$   such that $HH^\top=nI$, where $H^\top$ is
the transpose of $H$ and $I$ is the   identity matrix.
It is well known that the order of a Hadamard matrix is $1$, $2$,   or a multiple of $4$  \cite{pal}.
It is a  longstanding folklore conjecture that Hadamard matrices of order $n$ exist for any $n$ divisible by $4$.
The smallest order for which no known Hadamard matrix is  $668$   \cite{kha}.

A Hadamard matrix $H$ is said to be {\sl skew Hadamard} if $H+H^\top=2I$. Skew Hadamard matrices are equivalent to doubly regular tournaments \cite{bro}.
They are used to construct several combinatorial objects, such as association schemes, self-dual codes, strongly regular graphs, and more. It is conjectured that skew Hadamard matrices of order $n$ exist for any $n$ divisible by $4$ \cite{wall}.
The smallest unknown order of a skew Hadamard matrix is  $356$  \cite{doko276}.

A {\sl permutation matrix} (respectively,  {\sl signed permutation matrix})  is a matrix with entries in    $\{0, 1\}$  (respectively,   $\{-1, 0, 1\}$)  which has exactly one nonzero entry in each row and each column.
Two  matrices $X$ and $Y$ with entries in $\{-1, 1\}$  are  said to be {\sl {\sf H}-equivalent} if there exist two signed permutation matrices  $P$ and $Q$ such that $Y=PXQ$, otherwise, they are called  {\sl {\sf H}-inequivalent}.
Two      matrices $X$ and $Y$ with entries in $\{-1, 1\}$  are  said to be {\sl {\sf SH}-equivalent}  if there exists  a signed permutation matrix   $P$ such that $Y=P^\top XP$, otherwise, they are called  {\sl {\sf SH}-inequivalent}.

All  {\sf H}-inequivalent Hadamard matrices of orders up to $32$ have been classified \cite{khar}.
Also, all    {\sf H}-inequivalent and {\sf SH}-inequivalent skew Hadamard matrices of orders up to  $32$   have been classified \cite{hana}.
The resulting classification is shown in Table \ref{tab}. The complete classification of Hadamard matrices of  order $36$   seems to be difficult and
perhaps inaccessible.

\begin{table}[!htb]
\begin{center}
{\footnotesize\begin{tabular}{|c|c|c|c|c|c|c|c|c|c|c|}\hline
order & 1 & 2 & 4 & 8 & 12 & 16 & 20 & 24 & 28 & 32  \\ \hline
\texttt{\#} {\sf H}-inequivalent   Hadamard  matrices  & 1 & 1 & 1 & 1 & 1 & 5 & 3 & 60 & 487 & 13710027  \\ \hline
\texttt{\#} {\sf H}-inequivalent   skew Hadamard  matrices  & 1 & 1 & 1 & 1 & 1 &  2 &  2 &  16 & 54 & 6662  \\ \hline
\texttt{\#} {\sf SH}-inequivalent   skew Hadamard  matrices  & 1 & 1 & 1 & 1 & 1 & 2  & 2 & 16 & 65  & 7227   \\ \hline
\end{tabular}}
\caption{The number of  {\sf H}-inequivalent  Hadamard matrices,  {\sf H}-inequivalent  skew Hadamard matrices,  and  {\sf SH}-inequivalent  skew Hadamard matrices  of orders  up to  $32$.}\label{tab}
\end{center}
\end{table}

This paper determines the number of    {\sf SH}-inequivalent skew Hadamard matrices of order  $36$ for some types.
As an application, we study ternary self-dual codes and association schemes constructed from the skew Hadamard matrices of order $36$.

\section{Preliminaries}

In this section, we fix our notation and present some preliminary results.  Throughout the paper,  we denote all one vector of length  $r$  by    $\mathbb {1}_r$.
Let $H=[h_{uv}]$ be a Hadamard matrix of order $n$.
We know  from \cite{coo} that,  by a sequence of        column negations    and   column   permutations, every four  distinct  rows $i, j, k, \ell$
of $H$ may be transformed   to the   form
\begin{eqnarray}\label{defini}
\begin{array}{rrrrrrrrrr}
\mathbbm{1}_s & \mathbbm{1}_t & \mathbbm{1}_t & \mathbbm{1}_s & \mathbbm{1}_t & \mathbbm{1}_s & \mathbbm{1}_s & \mathbbm{1}_t \\
\mathbbm{1}_s & \mathbbm{1}_t & \mathbbm{1}_t & \mathbbm{1}_s & -\mathbbm{1}_t & -\mathbbm{1}_s & -\mathbbm{1}_s & -\mathbbm{1}_t \\
\mathbbm{1}_s & \mathbbm{1}_t & -\mathbbm{1}_t & -\mathbbm{1}_s & \mathbbm{1}_t & \mathbbm{1}_s & -\mathbbm{1}_s & -\mathbbm{1}_t \\
\mathbbm{1}_s & -\mathbbm{1}_t & \mathbbm{1}_t & -\mathbbm{1}_s & \mathbbm{1}_t & -\mathbbm{1}_s & \mathbbm{1}_s & -\mathbbm{1}_t
\end{array}\end{eqnarray}
for some uniquely determined $s, t$ with $s+t=\frac{n}{4}$. By negation of the last row in \eqref{defini} and then a suitable column permutation if necessary, we may assume that $s\geqslant t$ and so   $0\leqslant t\leqslant\lfloor\frac{n}{8}\rfloor$.
Following \cite{khloe}, we define   the {\sl type} of the four rows   $i, j, k, \ell$ as  $T_{ijk\ell}=t$. It is straightforward to   check  that   $T_{ijk\ell}=\frac{n-|P_{ijk\ell}|}{8}$, where $$P_{ijk\ell}=\sum_{r=1}^n h_{ir}h_{jr}h_{kr}h_{\ell r}.$$
This shows that `type' is an equivalence invariant, meaning that any negation of rows and columns and any permutation of columns leaves the type unchanged.
As  $H^\top$ is a Hadamard matrix, we may define `type' for any quadruple of columns of $H$ in a similar way. A Hadamard
matrix is of {\sl type}  $t$ if it has a quadruple of rows  of type $t$  and no quadruple of rows
of type less than $t$.
We refer to  \cite{MT} for more information about types of Hadamard matrices.

Below, we find a nice form for a quadruple of rows of a skew Hadamard matrix.
It is not hard  to verify that
any $4\times 4$   skew-symmetric   matrix with entries in $\{-1, 1\}$  and    with constant diagonal   is {\sf H}-equivalent to    one of the matrices
\begin{equation}\label{interAB}
A=\left[
\begin{array}{rrrr}
                                       1 & 1 & 1 & 1\\
                                       -1 & 1 & 1 & 1 \\
                                       -1& -1 & 1 & 1 \\
                                       -1 & -1 & -1 &1
                                     \end{array}
                                   \right] \quad  \text{ or }  \quad
                                   B=\left[
                                     \begin{array}{rrrr}
                                       1 & 1& 1 & 1 \\
                                       -1 & 1 & 1 & -1 \\
                                       -1 & -1 & 1 &  1 \\
                                      -1 &  1 & -1 & 1
\end{array}\right].
\end{equation}
Therefore,
by a sequence of  row  negations and row permutations,   and   simultaneously,  column  negations  and column  permutations, every quadruple  of    rows of   a skew Hadamard matrix   may be transformed   to one of  the   forms
\begin{eqnarray}\label{defskewA}\begin{array}{rrrrrrrrrrrrrrrrrr}
1& 1& 1& 1&   \mathbbm{1}_{s-1}  &   \mathbbm{1}_{t-1} &  \mathbbm{1}_t & \mathbbm{1}_{s-1} & \mathbbm{1}_t & \mathbbm{1}_s & \mathbbm{1}_s    & \mathbbm{1}_{t-1}  \\
-1& 1& 1& 1&   \mathbbm{1}_{s-1}   &   \mathbbm{1}_{t-1} &  \mathbbm{1}_t & \mathbbm{1}_{s-1} & -\mathbbm{1}_t & -\mathbbm{1}_s & -\mathbbm{1}_s    & -\mathbbm{1}_{t-1}  \\
-1& -1& 1& 1&   \mathbbm{1}_{s-1}   &   \mathbbm{1}_{t-1} &  -\mathbbm{1}_t & -\mathbbm{1}_{s-1} & \mathbbm{1}_t & \mathbbm{1}_s & -\mathbbm{1}_s    & -\mathbbm{1}_{t-1}  \\
-1& -1& -1& 1&   \mathbbm{1}_{s-1}   &   -\mathbbm{1}_{t-1} &  \mathbbm{1}_t  & -\mathbbm{1}_{s-1} & \mathbbm{1}_t & -\mathbbm{1}_s & \mathbbm{1}_s   & -\mathbbm{1}_{t-1}
\end{array}\end{eqnarray}
or
\begin{eqnarray}\label{defskewB}\begin{array}{rrrrrrrrrrrrrrrrrr}
1& 1& 1& 1&   \mathbbm{1}_s  &   \mathbbm{1}_{t-1} &  \mathbbm{1}_{t-1} & \mathbbm{1}_{s} & \mathbbm{1}_{t-1} & \mathbbm{1}_s & \mathbbm{1}_s    & \mathbbm{1}_{t-1}  \\
-1& 1& 1& -1&   \mathbbm{1}_s   &   \mathbbm{1}_{t-1} &  \mathbbm{1}_{t-1} & \mathbbm{1}_{s} & -\mathbbm{1}_{t-1} & -\mathbbm{1}_s & -\mathbbm{1}_s    & -\mathbbm{1}_{t-1}  \\
-1& -1& 1& 1&   \mathbbm{1}_s   &   \mathbbm{1}_{t-1} &  -\mathbbm{1}_{t-1} & -\mathbbm{1}_{s} & \mathbbm{1}_{t-1} & \mathbbm{1}_s & -\mathbbm{1}_s    & -\mathbbm{1}_{t-1}  \\
-1&  1& -1& 1&   \mathbbm{1}_s   &   -\mathbbm{1}_{t-1} &  \mathbbm{1}_{t-1}  & -\mathbbm{1}_{s} & \mathbbm{1}_{t-1} & -\mathbbm{1}_s & \mathbbm{1}_s   & -\mathbbm{1}_{t-1}
\end{array}\end{eqnarray}
for some   $s, t$  with  $s+t=\frac{n}{4}$. The quadruples of rows given in \eqref{defskewA} and \eqref{defskewB} are of type  $\min\{s, t\}$.
Letting
$$U=\left[
             \begin{array}{rrrr}
               0 & 0 & 0 & -1 \\
              1 & 0 & 0 &0 \\
               0 & 1 & 0 &0 \\
               0 & 0 & 1 &0 \\
             \end{array}
\right]$$
and $V=\mathrm{diag}(-I_{s-1}, I_{t-1}, -I_{t}, I_{s-1}, -I_{t}, I_{s}, -I_{s}, I_{t-1})$ and multiplying the  quadruple  of    rows given in  \eqref{defskewA} by $U$ from the left and   by  $U^\top\oplus V$ from the right, we may assume that $s\geqslant t$. This means the quadruple of rows given in \eqref{defskewA} is of type  $t$.
We define the {\sl skew type} of the quadruples of rows given in \eqref{defskewA}   as $(t, 1)$.
Also, we define the {\sl skew type} of  the    quadruples of rows  given in \eqref{defskewB}  as $(s, 2)$ if $s\leqslant t$  and as $(t, 0)$ if $s> t$.
Note that skew types $(0, 0)$ and $(0, 1)$ are not admissible. A skew Hadamard
matrix is of {\sl skew type}  $(t, \varepsilon)$ if it is a Hadamard matrix of    type $t$ containing      a quadruple of rows of  skew   type $(t, \varepsilon)$  and   no quadruples  of rows
of skew type   $(t, \varepsilon')$ with     $\varepsilon'<\varepsilon$.

For a    quadruple $\{i, j, k, \ell\}$ of rows of a  skew Hadamard matrix   $H=[h_{ij}]$, define   $$Q_{ijk\ell}=\sum_{r\in\{i, j, k, \ell\}}h_{ir}h_{jr}h_{kr}h_{\ell r}.$$
Note that    $Q_{ijk\ell}=0$ if and only if   the  $4\times 4$ principal submatrix corresponding to the     quadruple $\{i, j, k, \ell\}$ is     {\sf SH}-equivalent   to $A$,  given in \eqref{interAB}.
Moreover,  the parameter   $P_{ijk\ell}Q_{ijk\ell}$ is invariant under      negation and permutation of rows,     and simultaneously,   negation and permutation   of   columns
The following proposition describes the skew type of a quadruple of rows of a skew Hadamard matrix.  It is easily proved using the forms \eqref{defskewA} and \eqref{defskewB} and noting that   $P_{ijk\ell}Q_{ijk\ell}$ is an invariant parameter.

\begin{proposition}\label{des-type}
Let $H$ be a skew Hadamard matrix of order $n$.
Then, the following statements hold for every quadruple $\{i, j, k, \ell\}$ of rows of $H$ of type $t$.
\begin{itemize}
\item[{\rm (i)}] The skew type  of $\{i, j, k, \ell\}$    is    $(t, 0)$ if and only of $P_{ijk\ell}Q_{ijk\ell}<0$.
\item[{\rm (ii)}] The skew type of   $\{i, j, k, \ell\}$    is    $(t, 1)$ if and only of $Q_{ijk\ell}=0$.
\item[{\rm (iii)}]  The skew type  of  $\{i, j, k, \ell\}$    is    $(t, 2)$ if and only of $Q_{ijk\ell}\neq0$ and  $P_{ijk\ell}Q_{ijk\ell}\geqslant0$.
\end{itemize}
\end{proposition}

The following lemma shows the relation between the skew type of a quadruple of rows in a skew Hadamard matrix $H$ and the skew type of their corresponding rows in  $H^\top$.

\begin{lemma}\label{3part}
Let $H$ be a skew Hadamard matrix of order $n$.
Then, the following statements hold.
\begin{itemize}
\item[{\rm (i)}] If the quadruple $\{i, j, k, \ell\}$ of rows of $H$  is of skew type  $(t, 0)$, then the   quadruple $\{i, j, k, \ell\}$ of rows of  $H^\top$  is  of skew type  $(t-1, 2)$.
\item[{\rm (ii)}]  If the quadruple $\{i, j, k, \ell\}$ of rows of $H$  is of skew type  $(t, 1)$, then the   quadruple $\{i, j, k, \ell\}$  of rows of  $H^\top$  is also of skew type  $(t, 1)$.
\item[{\rm (iii)}]  If the quadruple $\{i, j, k, \ell\}$ of rows of $H$  is of skew type  $(s, 2)$, then the   quadruple $\{i, j, k, \ell\}$  of rows of  $H^\top$  is  of skew type
$$\left\{\begin{array}{ll}
(s-1, 2)     &   \mbox{ if } s=\frac{n}{8}\mbox{,}\\
\vspace{-1mm}\\
(s, 2)     &   \mbox{ if } s=\frac{n-4}{8}\mbox{,}\\
\vspace{-1mm}\\
(s+1, 2)     &   \mbox{ if } s=\frac{n}{8}-1\mbox{,}\\
\vspace{-1mm}\\
(s+1, 0) &   \mbox{   } \mbox{otherwise.}
\end{array}\right.
$$
\end{itemize}
\end{lemma}

\begin{proof}
To prove (i), let $\{i, j, k, \ell\}$  be a quadruple of  rows of $H$   of skew type  $(t, 0)$ and let $s=\frac{n}{4}-t$.
By \eqref{defskewB}, we have $T_{ijk\ell}=\frac{n-|P_{ijk\ell}|}{8}=\frac{4(s+t)-4|s-t+2|}{8}=t-1$  in $H^\top$, by noting that   $s>t$.
Since $P_{ijk\ell}Q_{ijk\ell}>0$ in   $H^\top$, it follows from  Proposition \ref{des-type}(iii) that the quadruple $\{i, j, k, \ell\}$ of rows of  $H^\top$  is  of skew type  $(t-1, 2)$.

For  (ii), let $\{i, j, k, \ell\}$  be a quadruple of  rows of $H$   of skew type  $(t, 1)$ and let $s=\frac{n}{4}-t$.
By \eqref{defskewA}, we have      $T_{ijk\ell}=\frac{n-|P_{ijk\ell}|}{8}=\frac{4(s+t)-4|s-t|}{8}=t$   in $H^\top$,   by noting that   $s\geqslant t$.
Since     $Q_{ijk\ell}=0$ in   $H^\top$, it follows from  Proposition \ref{des-type}(ii) that     the   quadruple $\{i, j, k, \ell\}$ of rows of  $H^\top$  is  of skew type  $(t, 1)$.

Finally, to establish (iii), let $\{i, j, k, \ell\}$  be a quadruple of  rows of $H$   of skew type  $(s, 2)$ and let $t=\frac{n}{4}-s$.
As $s\leqslant t$, it follows from   \eqref{defskewB} that
$$T_{ijk\ell}=\frac{n-|P_{ijk\ell}|}{8}=\frac{4(s+t)-4|s-t+2|}{8}=\left\{\begin{array}{ll}
s-1      &   \mbox{ if } s=\frac{n}{8}\mbox{,}\\
\vspace{-1mm}\\
s     &   \mbox{ if } s=\frac{n-4}{8}\mbox{,}\\
\vspace{-1mm}\\
s+1 &   \mbox{   } \mbox{otherwise,}
\end{array}\right.
$$
in $H^\top$.
If $s\in\{\frac{n}{8}, \frac{n-4}{8}, \frac{n}{8}-1\}$, then       $P_{ijk\ell}Q_{ijk\ell}\geqslant0$ in   $H^\top$,   and otherwise,       $P_{ijk\ell}Q_{ijk\ell}<0$ in   $H^\top$. As $Q_{ijk\ell}\neq0$  in   $H^\top$, the result follows from  Proposition \ref{des-type}.
\end{proof}

We recall the following two known lemmas on the types of a Hadamard matrix and its transpose.

\begin{lemma}[{\cite{kharT, khloe, MT}}]\label{Haad1}
There is no Hadamard matrix of order $n$ and  type $\lfloor\frac{n}{8}\rfloor$ if $n\notin\{4, 12\}$.
There is no Hadamard matrix of order $n$ and  type $0$ if $n\equiv 4
\pmod{8}$.
\end{lemma}

\begin{lemma}[\cite{kharT}]\label{Haad2}
Let  $H$ be a     Hadamard matrix of order $n$ having a quadruple of rows of type $1$.  Then, the type of $H^\top$ is     $0$ if $n\equiv 0
\pmod{8}$ and     $1$ if $n\equiv 4
\pmod{8}$.
\end{lemma}

\begin{theorem}\label{B+B-}
The following statements hold.
\begin{itemize}
\item[{\rm (i)}] There is no skew   Hadamard matrix of     skew type $(2, 0)$.
\item[{\rm (ii)}] There is no skew   Hadamard matrix of order $n$ and skew type $(1, 0)$ if $n\equiv 4 \pmod{8}$.
\item[{\rm (iii)}] There is no skew   Hadamard matrix of order $n$ and skew type $(1, 1)$ or $(1, 2)$ if $n\equiv 0 \pmod{8}$.
\end{itemize}
\end{theorem}

\begin{proof}
Let  $H$ be a  skew   Hadamard matrix of order $n$.
Suppose the skew type of $H$ is $(2, 0)$. Then,  Lemma \ref{3part}(i)  implies  that $H^\top$ has   a quadruple of rows of  skew type $(1, 2)$ and so $H^\top$ is of type  $0$ or $1$. If $H^\top$ is of type    $1$, then, by applying    Lemma  \ref{Haad2} for  $H^\top$, we deduce that   $H$ is of type $0$ or $1$ which is  a contradiction, since the     skew type of $H$ is   $(2, 0)$. Therefore,  $H^\top$ is of type  $0$ and hence   $n\equiv 0 \pmod{8}$ by Lemma \ref{Haad1}. Since the skew type of $H^\top$ is  $(0, 2)$, Lemma \ref{3part}(iii)  implies that $H$ has a quadruple of rows of type $1$, which is again a contradiction, as the skew type of $H$ is   $(2, 0)$.

Now, let $n\equiv 4
\pmod{8}$.  If   $H$ is  of   skew type $(1, 0)$, then Lemma \ref{3part}(i) implies  that $H^\top$ has   a quadruple of rows of  skew type $(0, 2)$ and so $H^\top$ is of type  $0$, contradicting  Lemma  \ref{Haad1}.

Finally,  suppose  that     $n\equiv 0
\pmod{8}$ and   $H$ is  of skew type $(1, 1)$ or $(1, 2)$.  It follows from  Lemma \ref{Haad1} that  $n\neq8$.   Also,   Lemma \ref{Haad2} yields that the type of  $H^\top$ is $0$, so $H^\top$  has a quadruple of  skew type $(0, 2)$ rows. So, it follows from  Lemma \ref{3part}(iii)    that $H$ has a quadruple of rows of skew type $(1, 0)$. As  $H$ is of type $1$, we conclude that the skew type of $H$ is $(1, 0)$,  a contradiction.
\end{proof}

To proceed, we must recall the following result on the type of quadruples of rows of a Hadamard matrix.

\begin{lemma}[\cite{MT}]\label{HxH}
Let $H$ be a  Hadamard matrix of order $4m$.
Fix three rows of $H$  and let $\kappa_t$ be the number of other rows of type $t$ with these three rows.
Then,
$$\sum_{t=0}^{\left\lfloor\frac{m}{2}\right\rfloor}\kappa_t(m-2t)^2=m^2.$$
\end{lemma}

\begin{lemma}\label{|A|}
Let $H$ be a  skew  Hadamard matrix of order $n$. Then, the   number of all $4\times 4$  principal submatrices of  $H$ which are  {\sf SH}-equivalent   to $A$ is $n(n-1)(n-2)(n-4)/32$, where $A$ is introduced in \eqref{interAB}.
\end{lemma}

\begin{proof}
For simplicity, let $n=4m$. It is easy to see that,
by a sequence of  row  negations and row permutations,   and   simultaneously,  column  negations  and column  permutations, every triple   of    rows   of   $H$    may be transformed   to    the   form
\begin{eqnarray*}\begin{array}{rrrrrrrrrrrrrrrrrr}
1& 1& 1&    \mathbbm{1}_{m-1}  &   \mathbbm{1}_{m-1} &  \mathbbm{1}_{m-1} & \mathbbm{1}_{m}   \\
-1& 1& 1&     \mathbbm{1}_{m-1}   &   \mathbbm{1}_{m-1} &  -\mathbbm{1}_{m-1} & -\mathbbm{1}_{m}    \\
-1& -1& 1&     \mathbbm{1}_{m-1}   &   -\mathbbm{1}_{m-1} &  -\mathbbm{1}_{m-1} & \mathbbm{1}_{m}
\end{array}.\end{eqnarray*}
One may   straightforwardly    check that    $Q_{123\ell}=0$  for the   quadruple  of rows  $\{1, 2, 3, \ell\}$  if and only if     $\ell\in\{4, \ldots, 3m\}$. This means that  the
quadruple  of rows  $\{1, 2, 3, \ell\}$   forms  a $4\times 4$  principal submatrix of  $H$ which is  {\sf SH}-equivalent   to $A$ if and only if     $\ell\in\{4, \ldots, 3m\}$. Therefore,  every triple of rows of $H$ is contained in exactly $3m-3$ quadruples of rows whose corresponding  $4\times 4$ principal submatrix is   {\sf SH}-equivalent to $A$.
By double counting, we derive that the number of all quadruples of rows of $H$  whose corresponding  $4\times 4$ principal submatrices are    {\sf SH}-equivalent to $A$  is equal to  $(3m-3){{n}\choose{3}}/4=n(n-1)(n-2)(n-4)/32$, as desired.
\end{proof}

\begin{theorem}\label{3-2-36}
There is no skew   Hadamard matrix of order  $36$ and skew type $(3, 2)$.
\end{theorem}

\begin{proof}
Suppose by way of contradiction that  $H$ is a  skew   Hadamard matrix of order $36$ and skew type    $(3, 2)$. So, every quadruple of rows of $H$ has type  $3$ or $4$. Using the notation of Lemma \ref{HxH}, we have $9\kappa_3^2+\kappa_4=81$. Since $\kappa_3+\kappa_4=33$, we obtain  that $\kappa_3=6$ and $\kappa_4=27$. By double counting, we derive that the number of all quadruples of rows of $H$ of type  $3$ is $6{{36}\choose{3}}/4=10710$ and therefore, the number of all quadruples of rows of $H$ of type  $4$ is ${{36}\choose{4}}-10710=48195$.

The possible skew types for any  quadruple of  rows of $H$ are   $(3, 2)$, $(4, 0)$, $(4, 1)$,  $(4, 2)$.
Using   Lemma \ref{3part}, the possible skew types for any  quadruple of  rows of $H^\top$ are   $(3, 2)$, $(4, 0)$, $(4, 1)$,  $(4, 2)$ and therefore, $H^\top$ is of skew type    $(3, 2)$.
Since the number of all quadruples of rows of $H$ of type  $3$ is $10710$,  the number of all quadruples of rows of $H$ of skew type  $(3, 2)$   is $10710$. Also, Lemma \ref{3part} implies that the number of all quadruples of rows of $H$ of skew type  $(4, 0)$ is equal to the number of all quadruples of rows of $H^{\top}$ of skew type  $(3, 2)$. Therefore,  the number of all quadruples of rows of $H$ of skew type  $(4, 0)$  is  $10710$.

The number of all quadruples of rows of $H$ of skew type $(4, 1)$ is equal to the number of all quadruples of rows of $H$  whose corresponding   $4\times 4$ principal submatrices are     {\sf SH}-equivalent to $A$ which equals  $42840$, by   Lemma \ref{|A|}. Now,    the  number of all quadruples of  rows of $H$ of skew type $(4, 2)$ is  $48195-(10710+42840)<0$, a contradiction.
\end{proof}

We end this section with the following lemma, which is crucial in our computational search.

\begin{lemma}\label{mod4}
Let  $H$ be a     Hadamard matrix of order $n$ whose first four rows are in the form \eqref{defini}.  Let $(h_1,  \ldots, h_n)$ be a row of $H$ other than the first four rows and define
$$x_1=\sum_{i=1}^{s}h_i, \qquad x_2=\sum_{i=s+2t+1}^{2s+2t}h_i, \qquad x_3=\sum_{i=2s+3t+1}^{3s+3t}h_i, \qquad  x_4=\sum_{i=3s+3t+1}^{4s+3t}h_i$$
and
$$y_1=\sum_{i=s+1}^{s+t}h_i, \qquad y_2=\sum_{i=s+t+1}^{s+2t}h_i, \qquad y_3=\sum_{i=2s+2t+1}^{2s+3t}h_i, \qquad  y_4=\sum_{i=4s+3t+1}^{n}h_i.$$
Then, $x_1+x_2+x_3-x_4=2s-\frac{n}{2}\pmod{4}$ and $y_1+y_2+y_3-y_4=2t-\frac{n}{2}\pmod{4}$.
\end{lemma}

\begin{proof}
Since $(h_1,  \ldots, h_n)$ is orthogonal to the four rows  given   in   \eqref{defini}, we obtain   that
\begin{align}
x_1+y_1+y_2+x_2+y_3+x_3+x_4+y_4&=0,\label{ell1}\\
x_1+y_1+y_2+x_2-y_3-x_3-x_4-y_4&=0,\label{ell2}\\
x_1+y_1-y_2-x_2+y_3+x_3-x_4-y_4&=0,\label{ell3}\\
x_1-y_1+y_2-x_2+y_3-x_3+x_4-y_4&=0 \label{ell4}.
\end{align}
By subtracting  \eqref{ell4} from the equality obtained by taking the sum of \eqref{ell1}--\eqref{ell3}, we conclude that $x_1+x_2+x_3-x_4=-2y_1$. Since   $2y_1=2t\pmod{4}$ and $s+t=\frac{n}{4}$,   one  deduces  that $x_1+x_2+x_3-x_4=2s-\frac{n}{2}\pmod{4}$.
Also, by summing \eqref{ell1}--\eqref{ell4}, we  get  $y_1+y_2+y_3-y_4=-2x_1$. Since $2x_1=2s\pmod{4}$ and $s+t=\frac{n}{4}$,   we get  $y_1+y_2+y_3-y_4=2t-\frac{n}{2}\pmod{4}$, the result follows.
\end{proof}

\section{Search using skew types}

In this section, we describe our search method for skew  Hadamard matrices, which uses the notion of skew types. Let $H$ be a skew  Hadamard matrix of order $n$ and skew type $(t, \varepsilon)$.
By Lemma  \ref{3part}, one can easily see that if $\varepsilon=0$, then $H^\top$ is of skew type $(t-1, 2)$. Therefore, the classification of skew  Hadamard matrices of skew type $(t,0)$ can be extracted from the
classification of skew  Hadamard matrices of skew type $(t-1,2)$. Hence, we may assume that  $\varepsilon=1,2$. Using  the forms \eqref{defskewA},  \eqref{defskewB} and by a permutation on  columns,
we may assume that the first four rows of   $H$ are of  the form \begin{eqnarray}\label{Aalg}\begin{array}{rrrrrrrrrrrrrrrrrr}
1& 1& 1& 1&   \mathbbm{1}_{t-1}  &   \mathbbm{1}_{t} &  \mathbbm{1}_t & \mathbbm{1}_{t-1} & \mathbbm{1}_{s-1} & \mathbbm{1}_{s-1} & \mathbbm{1}_s    & \mathbbm{1}_{s}  \\
-1& 1& 1& 1&   \mathbbm{1}_{t-1}   &   \mathbbm{1}_{t} &  -\mathbbm{1}_t & -\mathbbm{1}_{t-1} & \mathbbm{1}_{s-1} & \mathbbm{1}_{s-1} & -\mathbbm{1}_s    & -\mathbbm{1}_{s}  \\
-1& -1& 1& 1&   \mathbbm{1}_{t-1}   &   -\mathbbm{1}_{t} &  \mathbbm{1}_t & -\mathbbm{1}_{t-1} & \mathbbm{1}_{s-1} & -\mathbbm{1}_{s-1} & \mathbbm{1}_s    & -\mathbbm{1}_{s}  \\
-1& -1& -1& 1&   -\mathbbm{1}_{t-1}   &   \mathbbm{1}_{t} &  \mathbbm{1}_t  & -\mathbbm{1}_{t-1} & \mathbbm{1}_{s-1} & -\mathbbm{1}_{s-1} & -\mathbbm{1}_s   & \mathbbm{1}_{s}
\end{array}\end{eqnarray}
or
\begin{eqnarray}\label{Balg}\begin{array}{rrrrrrrrrrrrrrrrrr}
1& 1& 1& 1&   \mathbbm{1}_t  &   \mathbbm{1}_{t} &  \mathbbm{1}_{t} & \mathbbm{1}_{t} & \mathbbm{1}_{s-1} & \mathbbm{1}_{s-1} & \mathbbm{1}_{s-1}    & \mathbbm{1}_{s-1}  \\
-1& 1& 1& -1&   \mathbbm{1}_t   &   \mathbbm{1}_{t} &  -\mathbbm{1}_{t} & -\mathbbm{1}_{t} & \mathbbm{1}_{s-1} & \mathbbm{1}_{s-1} & -\mathbbm{1}_{s-1}    & -\mathbbm{1}_{s-1}  \\
-1& -1& 1& 1&   \mathbbm{1}_t   &   -\mathbbm{1}_{t} &  \mathbbm{1}_{t} & -\mathbbm{1}_{t} & \mathbbm{1}_{s-1} & -\mathbbm{1}_{s-1} & \mathbbm{1}_{s-1}    & -\mathbbm{1}_{s-1}  \\
-1&  1& -1& 1&   \mathbbm{1}_t   &   -\mathbbm{1}_{t} &  -\mathbbm{1}_{t}  & \mathbbm{1}_{t} & -\mathbbm{1}_{s-1} & \mathbbm{1}_{s-1} & \mathbbm{1}_{s-1}   & -\mathbbm{1}_{s-1}
\end{array}\end{eqnarray}
where $s=\frac{n}{4}-t\geqslant t$.

By applying  Lemma \ref{mod4}, if the first four rows of   $H$ are of the form    \eqref{Aalg}, then the column $4t+2$ of $H$  is uniquely determined from the columns $1,3,5,\dots,4t+1$ of $H$ and hence the row  $4t+2$ of $H$  is uniquely determined from the rows $1,3,5,\dots,4t+1$ of $H$.
Also, using    Lemma \ref{mod4},  if the first four rows of   $H$ are of the form    \eqref{Balg}, then the column $4t+4$ of $H$  is uniquely determined from the columns $5,\dots,4t+3$ of $H$ and thus the row $4t+4$ of $H$  is uniquely determined from the rows $5,\dots,4t+3$ of $H$.

To classify skew Hadamard matrices of skew type $(t, \varepsilon)$ for $\varepsilon=1,2$, we first fix the first four rows according to one of the forms \eqref{Aalg} and  \eqref{Balg}.     We then use a backtrack algorithm to construct the next rows one by one.
By the above paragraph, when we reach the row  $4t+2$ or $4t+4$,  depending on the form of the first four rows, we have pruning criteria for the backtracking. We compute this row from the preceding rows. If there is no solution,  we backtrack. Otherwise, we proceed to the next step.

During the backtracking algorithm,  we perform some checks. Regarding the number of solutions permitted in the primary steps of backtracking, we do the   {\sf SH}-equivalence test on the obtained partial skew Hadamard matrices.
Also,  we check for the skew type of partial skew Hadamard matrices whenever it is computationally not time-consuming.

After we obtain all solutions, we need to check them for the {\sf SH}-equivalence. It could take a long time because of the large number of solutions. To do this, we partition the set of solutions according to
their skew profiles, which count the number of quadruples of rows of different skew types. Then,  in each class, which is usually much smaller, we do the  {\sf SH}-equivalence test.

In order $36$, in view of  Lemma \ref{Haad1}, there is no skew Hadamard matrix of skew type $(0, \varepsilon)$ and $(4, \varepsilon)$ for any $\epsilon$.
By Lemma \ref{B+B-}, there is no such matrix of skew type $(1, 0)$ and $(2, 0)$. We also have no matrix of type $(3, 2)$, using  Lemma \ref{3-2-36}.  So,  we need to find solutions for skew types $(1, 1)$, $(1, 2)$, $(2, 1)$, $(2, 2)$, $(3, 0)$ and  $(3, 1)$.
We have  run the program for skew types $(1, 1)$, $(1, 2)$  and  $(2, 1)$.
The two skew types $(1, 1)$ and  $(1, 2)$ are fast enough to be run on a  single PC  for a few days.
The computation time for the skew type $(2,1)$ was about twenty days on a cluster of 48 cores. We also implemented some parts of the algorithm twice with different codes.
The remaining cases $ (2, 2) $, $ (3, 0) $ and $ (3, 1) $ need much more computational resources. We also run the program for all orders smaller than $36$. The results, which are given in Table \ref{tabskewtype},   confirm
the numbers shown in Table  \ref{tab}.  The obtained list of skew Hadamard matrices of order $36$ is available electronically at \cite{harada}.
We summarize our main result in the following proposition.

\begin{proposition}\label{SHskewthm}
There are at least     $157132$  {\sf SH}-inequivalent skew Hadamard matrices of order  $36$.
\end{proposition}

\begin{table}[!htb]
\begin{center}
{\footnotesize
\begin{tabular}{|c|c|c|c|c|c|c|c|c|}\hline
\backslashbox{{\footnotesize Order}}{{\footnotesize Skew Type}}
&{(0, 2)}&{(1, 0)}&{(1, 1)}&{(1, 2)}&{(2, 1)}&{(2, 2)}&{(3, 0)}&{(3, 1)}\\\hline
4 &1&0&0&0&0&0&0&0\\\hline
8 &1&0&0&0&0&0&0&0\\\hline
12 &0&0&1&0&0&0&0&0\\\hline
16 &2&0&0&0&0&0&0&0\\\hline
20 &0&0&1&1&0&0&0&0\\\hline
24 &14&1&0&0&1&0&0&0\\\hline
28 &0&0&43&21&0&1&0&0\\\hline
32 &6903&283&0&0&40&0&0&1\\\hline
36 &0&0&23260&123326&10546&\texttt{?}&\texttt{?}&\texttt{?}\\\hline
\end{tabular}
}
\caption{The number of    {\sf SH}-inequivalent  skew Hadamard matrices   of orders  up to   $36$ for all possible   skew types.}\label{tabskewtype}
\end{center}
\end{table}

\section{Search through the canonical method}

This section uses the orderly generation method to classify skew Hadamard matrices of order  $36$.
The technique was independently
introduced in      \cite{FAR} and \cite{REA}.
The method consists of two parallel subroutines. These are the construction of objects and
the rejection of equivalence copies.
The most natural and widely used method for the construction phase is backtracking, which has
quite an old history.
We use the notion of `canonical form' to reject equivalence copies.
A canonical form is a special representative for each equivalence class. The canonical forms are constructed
step by step using backtracking, and
the canonicity test is carried out at each step.

Here, we present the canonical form for skew Hadamard matrices. For a skew Hadamard matrix $H=[h_{ij}]$ of  order $n$,  assign a  vector
$$v(H)=\big(h_{1,2}, h_{1,3}, h_{2,3}, h_{1,4}, h_{2,4}, h_{3,4}, \ldots, h_{n-1,n}\big).$$
We say that $H$ is in the {\sl canonical form} if
$v(P^\top HP) \preccurlyeq v(H)$
for any signed permutation matrix $P$, where $\preccurlyeq$ denotes  the lexicographical order.

We used an orderly algorithm with backtracking and the above canonical form to generate  {\sf SH}-equivalence classes of
skew Hadamard matrices of order $36$. Using this orderly algorithm, we first find all solutions for the   $14\times 14$ leading principal submatrices in canonical form.
In total, there are $80122802$ solutions for  $14\times 14$ matrices. For each solution, we find all $14\times 36$ partial Hadamard matrices constituting the first  $14$ rows of desired skew Hadamard matrix of order $36$  using an exhaustive search.
Then,  for each obtained  $14\times 36$  partial Hadamard matrix and  $i=15, \ldots, 36$, we complete
every row $i$   by solving a system of linear equations coming from the orthogonality row $i$ to the first 14 rows.  To speed up the computation, we solve the system of linear equations in modulo   $2$.
Finally, for the complete solutions, we do the canonicity test.
We have not been able to run the computation for all $80122802$  solutions. It is done for about $10$ million of them.
The approximate computation time was three months on a cluster of 160 cores. We found $157132$ skew Hadamard matrices of order $36$ up to the  {\sf SH}-equivalence.
The list matches the list obtained in the previous section.

Denote by $J_k$   the skew-symmetric matrix of order  $k$    whose all entries on and above the main diagonal are $1$.
For example, the matrix $A$ given in \eqref{interAB} equals $J_4$.
The number of solutions for the $8\times 8$  leading principal submatrices of skew Hadamard matrices of order  $36$  in the canonical form is   $9$.
Our computation shows that only one of these $9$  solutions, $J_8$,   is
extendable to skew Hadamard matrices.  So, every skew Hadamard matrix of order $ 36 $ in canonical form contains $ J_8$.
Note that $J_8$ has $11$  feasible extensions to $9\times 9$ submatrices of skew Hadamard matrices of order  $36$. Let us show these solutions by $S_0=J_9$, $S_1,\ldots,S_{10}$ from the largest to smallest in the
canonical form. Our results show that $S_6,\ldots,S_{10}$ have no extensions. For $S_0$ and $S_1$,  we found the complete list of extensions. However,  the search is not complete for $S_2$. Also, we have not done anything for $S_3$, $S_4$  and $S_5$.

Inspired by the previous paragraph, we may ask for the maximum number  $k(n)$ such that $J_{k(n)}$ appears as a principal submatrix of a skew Hadamard matrix of order $n$.
This is related to  a  problem posed in \cite{erd}: What is the maximum number $t(n)$ such that every tournament of order $n$   contains the  transitive
tournament of order $t(n)$ as a subtournament?
The following theorem presents an upper bound on $k(n)$, which is especially tight for $n=36$.

\begin{theorem}
For any $n$, $$k(n)\leqslant\left\lfloor\frac{\pi}{2\cot^{-1}\big(\sqrt{n-1}\big)}\right\rfloor.$$
\end{theorem}

\begin{proof}
Let $H$ be a skew Hadamard matrix of order $n$. It is easy to see  that  $A=\mathbbm{i}(I-H)$ is a Hermitian matrix with $\frac{n}{2}$ eigenvalues $\sqrt{n-1}$ and  $\frac{n}{2}$ eigenvalues $-\sqrt{n-1}$, where  $\mathbbm{i}$ is the unit imaginary complex number. Let $k=k(n)$ and $B=\mathbbm{i}(I-J_k)$.  As   $J_k$ is a principal submatrix of $H$,  $B$   is a principal submatrix of $A$. Since  $B$ is a negacirculant matrix, it straightforwardly follows from a result of \cite{nega} that the eigenvalues of $B$ are $\cot(\frac{\pi\ell}{2k})$ for  $\ell=1, 3, \ldots, 2k-1$. By the interlacing theorem, the largest eigenvalue of $B$ is less than or equal to the  largest eigenvalue of $A$, that is,    $\cot(\frac{\pi}{2k})\leqslant\sqrt{n-1}$. This yields the desired inequality.
\end{proof}

\begin{remark}
Whether or not our list of skew Hadamard matrices of order  $36$  is complete is a question. Referring to Table \ref{tabskewtype}, we strongly feel that no skew Hadamard matrices of order $36$ exist with skew types $(2, 2)$, $(3, 0)$ and  $(3, 1)$.
\end{remark}

\section{Ternary codes}

Self-dual codes are one of the most interesting classes of codes.
This interest is justified by many combinatorial objects
and algebraic objects related to self-dual codes,  see, for example,    \cite{RS-Handbook}.
In this section, we classify ternary self-dual codes constructed from
the $157132$ skew Hadamard matrices of order $36$ given in Table \ref{tabskewtype}.
We also study unimodular lattices and $1$-designs built from the ternary near-extremal self-dual codes.

\subsection{Ternary self-dual codes}

Let $\FF_3=\{0,1,2\}$ denote the finite field of order $3$.
A {\sl ternary} $[n,k]$ {\sl code} $C$ is a $k$-dimensional vector subspace of $\FF_3^n$.
The parameters $n$ and $k$ are called the {\sl length} and {\sl dimension} of $C$, respectively.
A ternary {\sl  self-dual} code $C$ of length $n$
is a ternary $[n,n/2]$ code satisfying
$C=C^\perp$, where $C^\perp$ is the dual code  of $C$ under  the standard inner product.
It was shown in \cite{MS-bound} that
the minimum weight $d$ of a ternary self-dual code of length $n$
is bounded by $d\leqslant 3 \lfloor n/12 \rfloor+3$.
If $d=3\lfloor n/12 \rfloor+3$ (respectively,  $d=3\lfloor n/12 \rfloor$),
then the code is called {\sl extremal} (respectively,  {\sl near-extremal}).
For length $36$, the Pless symmetry code is a currently known extremal ternary self-dual code.

Two ternary codes $C$ and $C'$ are said to be {\sl equivalent} if there exists a
monomial matrix $P$ over $\FF_3$ with $C' = \{ x P \, | \,  x \in C\}$,
otherwise, they are called  {\sl inequivalent}.
Here, a monomial matrix is a
matrix with entries in $\FF_3$  with exactly one nonzero entry in each row and each column.
Let $H$ be a Hadamard matrix of order $n$.
Throughout this section, let $C(H)$ denote the ternary code generated by the rows of $H$,
where the entries of the matrix are regarded as elements of $\FF_3$.
It is trivial that $C(H)$ and $C(K)$ are equivalent if
$H$ and $K$ are {\sf SH}-equivalent skew Hadamard matrices.

Although the proof of the following lemma is somewhat trivial,
we give it for the sake of completeness.

\begin{lemma}
Let $H$ be a skew Hadamard matrix of order $n$.
If $n \equiv 0 \pmod{12}$, then $C(H)$ is self-dual.
\end{lemma}

\begin{proof}
Since  $n \equiv 0 \pmod{12}$, $HH^\top \equiv   0_n  \pmod 3$,
where $0_n$ denotes the $n\times n$ zero matrix.
This implies that  $C(H) \subseteq C(H)^\perp$.
From  $\dim C(H) + \dim  C(H)^\perp=n$, we find  that $\dim C(H)\leqslant \frac{n}{2}$.
On the other hand,
since $H+H^\top=2I$,
$$n=\rank_3(2I)=\rank_3\left(H+H^\top\right)\leqslant \rank_3(H)+\rank_3\left(H^\top\right) =2\rank_3(H),$$
where $\rank_{3}(X)$ denotes the  rank of  the matrix  $X$ regarded over  $\FF_3$.
Thus, $\rank_3(H)\geqslant \frac{n}{2}$ which yields  that $\dim C(H)\geqslant \frac{n}{2}$.
Therefore,  $\dim C(H)=  \frac{n}{2}$ and   so $C(H) = C(H)^\perp$, meaning that    $C(H)$ is self-dual.
\end{proof}

We describe how to classify ternary self-dual codes of length $36$ constructed from
the $157132$ skew Hadamard matrices of order $36$ given in Table \ref{tabskewtype}.
By comparing pairs $(A_6,A_9)$, the $157132$  ternary self-dual codes are divided into $829$ classes,
where $A_w$ denotes the number of codewords of weight $w$.
For each class, to test the equivalence of ternary self-dual codes,
we used the method given in Section 7.3.3 of  \cite{KO} as follows.
For a ternary self-dual code $C$, define
the vertex-colored graph $\mathnormal{\Gamma}(C)$ with vertex set
$C_9 \cup (\{1,2,\dots,36\}\times (\FF_3\setminus\{0\}))$
and edge set
$\{\{c,(i,c_i)\} \, | \,  c=(c_{1}, \ldots,c_{36}) \in C_9, i\in \{1, \ldots,  36\}   \text{ and }    c_i \ne 0\}
\cup \{\{(i,y),(i,2y)\} \, | \,
i\in \{1, \ldots,  36\}   \text{ and }     y \in \FF_3\setminus\{0\}\}$,
where $C_9$ denotes the set of codewords of weight $9$ in $C$.
Suppose that two ternary self-dual codes $C$ and $C'$ are generated by
codewords of weight $9$.
Then, $C$ and $C'$ are equivalent if and only if $\mathnormal{\Gamma}(C)$ and $\mathnormal{\Gamma}(C')$  are isomorphic.
We verified that the $157129$ ternary self-dual codes of length $36$
are generated by codewords of weight $9$  and the remaining three codes
are not generated by codewords of weight $9$.
For the three codes, we have that $(A_6,A_9)=(36, 464)$, $(96, 224)$ and $(96, 1520)$.
Thus, each of the three codes is inequivalent to the other codes.
We used {\sf nauty} \cite{nauty}
for vertex-colored graph isomorphism testing corresponding to the $157129$ ternary self-dual codes.
Then, we have the following classification of
ternary self-dual codes of length $36$ constructed from
the $157132$ skew Hadamard matrices of order $36$.

\begin{proposition}\label{code}
$153979$ inequivalent ternary self-dual codes of length $36$ are constructed from
the $157132$ skew Hadamard matrices of order $36$ given in Table \ref{tabskewtype}.
$20848$ of them have minimum weight $9$ and the others have minimum weight $6$.
\end{proposition}

The inequivalence of the above $153979$ ternary self-dual codes was verified independently by {\sf Magma} \cite{Magma}.
The obtained list of the ternary self-dual codes of length $36$ is available electronically at \cite{harada}.

\subsection{Ternary near-extremal self-dual codes}

Here, we concentrate on the $20848$ ternary near-extremal self-dual codes given in Proposition \ref{code}.
Recently, a restriction on the weight enumerators of ternary near-extremal self-dual codes of
lengths divisible by $12$ has been given in \cite{AH}.
The possible weight enumerators of ternary  near-extremal self-dual codes of length $36$
are
\begin{align*}
W_\alpha&=
1
+ \alpha y^9
+( 42840 - 9 \alpha)y^{12}
+( 1400256 + 36 \alpha)y^{15}
+( 18452280- 84 \alpha)y^{18}\\&
+( 90370368 + 126 \alpha)y^{21}
+( 162663480 - 126 \alpha)y^{24}
+( 97808480 + 84 \alpha)y^{27}
\\&
+( 16210656 - 36 \alpha)y^{30}
+( 471240 + 9 \alpha)y^{33}
+( 888 - \alpha)y^{36},
\end{align*}
where $\alpha$ is an integer with   $1\leqslant  \alpha \leqslant 888$ and $\alpha \equiv  0 \pmod 8$.
For the ternary near-extremal self-dual codes given in Proposition \ref{code},
$\alpha$ in their weight enumerators $W_\alpha$ are listed in Table \ref{tabcode}.
In addition, the numbers $N_\alpha$ of the ternary near-extremal self-dual codes
having weight enumerator $W_\alpha$ are listed in Table \ref{tabcode}.

\begin{table}[!htb]
\begin{center}
{\footnotesize
\begin{tabular}{|c|c|c|c|c|c|}\hline
$(\alpha,N_\alpha )$&
$(\alpha,N_\alpha )$&
$(\alpha,N_\alpha )$&
$(\alpha,N_\alpha )$&
$(\alpha,N_\alpha )$&
$(\alpha,N_\alpha )$
\\
\hline
$( 72,  1)$ &$(224,  25)$ &$(288, 1105)$ &$(352, 1352)$ &$(416,   83)$ &$(480,  2)$ \\
$(144,  1)$ &$(232,  41)$ &$(296, 1451)$ &$(360, 1123)$ &$(424,   49)$ &$(488,  2)$ \\
$(168,  3)$ &$(240, 105)$ &$(304, 1580)$ &$(368,  870)$ &$(432,   34)$ &$(496,  1)$ \\
$(176,  1)$ &$(248, 153)$ &$(312, 1860)$ &$(376,  679)$ &$(440,   26)$ &$(512,  1)$ \\
$(184,  1)$ &$(256, 263)$ &$(320, 1750)$ &$(384,  466)$ &$(448,    9)$ &$(544,  1)$ \\
$(200,  3)$ &$(264, 409)$ &$(328, 1894)$ &$(392,  342)$ &$(456,   15)$ &$(600,  1)$ \\
$(208,  9)$ &$(272, 577)$ &$(336, 1762)$ &$(400,  200)$ &$(464,    8)$ & \\
$(216, 13)$ &$(280, 864)$ &$(344, 1566)$ &$(408,  142)$ &$(472,    5)$ & \\
\hline
\end{tabular}
}
\caption{The number $N_\alpha$ of the ternary near-extremal self-dual codes.}
\label{tabcode}
\end{center}
\end{table}

A (Euclidean) lattice $L \subseteq \mathbbmsl{R}^n$ in dimension $n$ is called  {\sl unimodular} if $L = L^{*}$,
where $L^{*}$ is the dual lattice of $L$ under the standard inner product.
The {\sl minimum norm} of a unimodular lattice $L$ is the smallest
the norm among all nonzero vectors of $L$.
The {\sl kissing number} of $L$ is the number of vectors of minimum norm in $L$.
Two lattices $L$ and $L'$ are said to be {\sl isomorphic}
if there exists an orthogonal matrix $A$ with $L' = \{xA \, | \,  x \in L\}$,
otherwise, they are called  {\sl non-isomorphic}.
Let $C$ be a ternary self-dual code of length $n$.
It is known that
\[
A_{3}(C) =
\left\{\left.\frac{1}{\sqrt{3}}\big(x_1, \ldots,x_n\big) \in \mathbbmsl{Z}^n \, \right| \,
\Big(x_1  \hspace{-3.5mm} \pmod 3,  \ldots, x_n  \hspace{-3.5mm} \pmod 3\Big)\in C\right\}
\]
is an unimodular lattice in dimension $n$.
This construction of lattices is well known as {\sl Construction A}.
Let $C$ be a ternary near-extremal self-dual code of length $36$ having weight enumerator $W_\alpha$.
Then, it follows that $A_3(C)$ has minimum norm $3$ and kissing number $\alpha+72$.
We verified by {\sf Magma} \cite{Magma} that the $20848$
unimodular lattices, which are constructed from the $20848$ inequivalent
ternary near-extremal self-dual codes given in Proposition \ref{code},
are non-isomorphic.
This also verifies the inequivalence of the $20848$ ternary near-extremal self-dual codes.

Let $C$ be a ternary near-extremal self-dual code of length $36$ having weight enumerator $W_\alpha$.
Then, the supports of codewords of weight $9$ in $C$ form a $1$-$(36,9, \alpha/8)$ design   \cite{MMN}.
Let $D_i=(X_i,\mathcal{B}_i)$ be a $1$-design,
where $X_i$ is the set of points and $\mathcal{B}_i$ is the collection of blocks of $D_i$  for $i=1,2$.
Two $1$-designs $D_1$ and $D_2$ are said to be {\sl isomorphic} if there exists a bijection
$\phi:X_1 \to X_2$ such that
if we rename every point $x$ of $X_1$ by $\phi(x)$ then $\mathcal{B}_1$ is transformed into $\mathcal{B}_2$,
otherwise, they are called  {\sl non-isomorphic}.
We verified by {\sf Magma} \cite{Magma} that these $20848$ $1$-designs, which are constructed from the $20848$ inequivalent
ternary near-extremal self-dual codes  given in Proposition \ref{code},
are non-isomorphic.
This also verifies the inequivalence of the $20848$ ternary near-extremal self-dual codes.

\section{Association schemes}

In this section, we construct association schemes of order $35$ and class $2$ from skew Hadamard matrices of order $36$.
We define association schemes in matrix form.
An {\sl association scheme} of order $n$ and class $2$ is a set of nonzero $n \times n$
matrices $\{A_0 , A_1, A_2\}$ with entries in $\{0, 1\}$ satisfying
\begin{itemize}
\item[(i)] $A_0 = I$;
\item[(ii)] $A_0+A_1+A_2=J$, where $J$ denotes the all one matrix;
\item[(iii)] For every  $i,j \in \{0,1,2\}$, $A_i^\top \in \{A_0,A_1,A_2\}$ and $A_iA_j$ is a linear combination of  $A_0, A_1, A_2$.
\end{itemize}
An association scheme $\mathscr{A}=\{A_0 , A_1,A_2\}$ is said to be {\sl symmetric} if all matrices $A_i$ are symmetric.
Two association schemes $\mathscr{A}$ and $\mathscr{B}$ are said to be
{\sl isomorphic} if there exists a permutation matrix $P$ such that $\mathscr{B} = P^{-1}\mathscr{A}P$,
otherwise, they are called  {\sl non-isomorphic}.

For any $n \equiv 0 \pmod{4}$, there is a known correspondence between skew Hadamard matrices of order $n$ and nonsymmetric
association schemes of order $n-1$ and class $2$ \cite{hana}.
Let $H$ be a skew Hadamard matrix $H$ of order $n$.
Denote by  $D_i$   the diagonal matrix whose diagonal entries are the $i$th row of $H$ and denote by
$R_i$    the $(n-1) \times (n-1)$ matrix obtained from $D_i^{-1} H D_i$ by deleting the $i$th row and column.
Define
\[
A_0 = I, \, A_1 =\frac{1}{2}(J-2I + R_i) \, \text{ and } \, A_2 = \frac{1}{2}(J-R_i).
\]
Then, $\mathscr{A}(H,i) = \{A_0, A_1, A_2\}$ is an association scheme of order $n-1$ and class $2$.

Let $H_1, \ldots,H_{157132}$ denote the $157132$ skew Hadamard matrices of order $36$ given in Table \ref{tabskewtype}.
Then,
\[
\mathscr{A}(35)=
\big\{\mathscr{A}(H_i,j)  \, \big| \,  i\in\{1, \ldots,157132\}  \, \text{ and } \, j \in \{ 1, \ldots,36\}\big\}.
\]
is a set of nonsymmetric association schemes of order $35$ and class $2$.
We describe how to classify  nonsymmetric association schemes $\mathscr{A}(H_i,j)$ constructed from $H_i\in\{H_1, \ldots,H_{157132}\}$.
Let $\mathscr{A}=\{A_0,A_1,A_2\}$ and $\mathscr{B}=\{B_0,B_1,B_2\}$ be
nonsymmetric association schemes of order $35$ and class $2$.
Let $\mathrm{dg}(M)$ denote the digraph corresponding to a $35 \times 35$ matrix $M$ with entries in $\{0, 1\}$.
Then, $\mathscr{A}$ and $\mathscr{B}$ are isomorphic if and only if
$\mathrm{dg}(A_1)\simeq \mathrm{dg}(B_1)$ or $\mathrm{dg}(A_1) \simeq \mathrm{dg}(B_2)$, where
$X\simeq Y$ means that two digraphs  $X$ and $Y$ are isomorphic \cite{hana}.
Let $\mathrm{clg}(A)$ denote the canonically labeled graph corresponding to  $\mathrm{dg}(A)$, see   \cite{KO, nauty}
for the definition.
It is known that $\mathrm{clg}(A)=\mathrm{clg}(B)$ if and only if $\mathrm{dg}(A) \simeq \mathrm{dg}(B)$  \cite{KO, nauty}.
Then, $\mathscr{A}$ and $\mathscr{B}$ are isomorphic
if and only if $\{\mathrm{clg}(A_1),\mathrm{clg}(A_2)\}=\{\mathrm{clg}(B_1),\mathrm{clg}(B_2)\}$.
Considering canonically labeled graphs significantly reduces the amount of our computation.
In this way, by comparing canonically labeled graphs $\{\mathrm{clg}(A_1),\mathrm{clg}(A_2)\}$ in $\mathscr{A}=\{A_0,A_1,A_2\}$,
we determined by {\sf Magma} \cite{Magma} that $\mathscr{A}(35)$
contains $2793032$ non-isomorphic nonsymmetric association schemes.
This calculation was also done by using {\sf nauty} \cite{nauty}.

\begin{proposition}
At least $2793032$ non-isomorphic nonsymmetric association schemes of order $35$ and class $2$ exist.
\end{proposition}

The  obtained list of the non-isomorphic nonsymmetric association schemes of order $35$ and class $2$
is available electronically at \cite{harada}.

\section*{Acknowledgments}

Masaaki Harada was supported by JSPS KAKENHI under grant numbers 19H01802 and 23H01087.
Hadi Kharaghani was supported by the Natural Sciences and Engineering Research Council of Canada (NSERC).
Ali Mohammadian  was    supported by the    Natural Science Foundation of Anhui Province  with  grant identifier 2008085MA03 and   by   the National Natural Science Foundation of China with  grant number 12171002.


\begin{thebibliography}{99}




\bibitem{AH}
{\sc M. Araya} and {\sc M. Harada}, Some restrictions on the weight enumerators of near-extremal ternary self-dual codes and quaternary Hermitian self-dual codes,
{\sl Des.  Codes Cryptogr.}  {\bf 91}  (2023),   1813--1843.


\bibitem{Magma}
{\sc W. Bosma}, {\sc J. Cannon} and {\sc C. Playoust}, The Magma algebra system I: The user language, {\sl J. Symbolic Comput.} {\bf 24} (1997), 235--265.


\bibitem{bro}
{\sc E. Brown} and {\sc K.B. Reid},  Doubly regular tournaments are equivalent to skew Hadamard matrices, {\sl J. Combin. Theory Ser. A} {\bf 12} (1972), 332--338.


\bibitem{coo}
{\sc J. Cooper}, {\sc J.  Milas}  and  {\sc W.D.  Wallis}, Hadamard equivalence,  {\sl Combinatorial Mathematics},  in: Lecture Notes in Math., vol. 686, Springer,  pp.  126--135,  1978.


\bibitem{doko276}
{\sc D.\v{Z}. \DH okovi\'{c}}, Skew-Hadamard matrices of order $276$, available at:  \url{http://arxiv.org/pdf/2301.02751.pdf}.


\bibitem{erd}
{\sc P. Erd\H{o}s} and {\sc L.  Moser},   On the representation of directed graphs as unions of orderings, {\sl Magyar Tud.  Akad.  Mat.  Kutat\'{o} Int.  K\"{o}zl.} {\bf 9} (1964), 125--132.


\bibitem{FAR}
{\sc I.A. Farad\v{z}ev},   Constructive enumeration of combinatorial objects, in: {\sl Probl\`{e}mes combinatoires et th\'{e}orie des graphes} (Colloq.  Internat.  CNRS, Univ.  Orsay, Orsay, 1976),  Colloq.  Internat. CNRS, 260, CNRS, Paris,  pp.  131--135, 1978.


\bibitem{hana}
{\sc A. Hanaki}, {\sc  H.  Kharaghani}, {\sc  A.  Mohammadian} and  {\sc B.  Tayfeh-Rezaie},
Classification of skew-Hadamard matrices of order $32$ and association schemes of order $31$, {\sl J. Combin.  Des.}  {\bf 28} (2020),  421--427.


\bibitem{harada}
{\sc M. Harada}, Codes and association schemes, available  at: \url{http://www.math.is.tohoku.ac.jp/~mharada/F3-H36/}.


\bibitem{KO}
{\sc P. Kaski} and  {\sc P.R.J. \"Osterg\aa rd}, {\sl Classification Algorithms for Codes and Designs}, in; Algorithms Comput. Math., 15, Springer, Berlin, 2006.


\bibitem{khar}
{\sc H. Kharaghani} and {\sc B. Tayfeh-Rezaie},   Hadamard matrices of order $32$,  {\sl J. Combin.  Des.}  {\bf 21} (2013),   212--221.


\bibitem{kharT}
{\sc H. Kharaghani} and {\sc B. Tayfeh-Rezaie},    On the classification of Hadamard matrices of order $32$, {\sl  J. Combin.  Des.} {\bf 18} (2010),  328--336.


\bibitem{kha}
{\sc H. Kharaghani} and {\sc B. Tayfeh-Rezaie},  A Hadamard matrix of order $428$, {\sl  J. Combin.  Des.} {\bf  13} (2005), 435--440.


\bibitem{khloe}
{\sc H. Kimura}, Classification of Hadamard matrices of order $28$, {\sl Discrete Math.}  {\bf 133}  (1994), 171--180.


\bibitem{nega}
{\sc E.G.  Ko\c{c}er}, Circulant, negacyclic and semicirculant matrices with the modified Pell, Jacobsthal and Jacobsthal-Lucas numbers, {\sl Hacet.  J. Math.  Stat.} {\bf 36} (2007), 133--142.


\bibitem{MS-bound}
{\sc C.L. Mallows} and {\sc N.J.A. Sloane}, An upper bound for self-dual codes, {\sl Inform. and Comput.} {\bf 22} (1973), 188--200.


\bibitem{nauty}{\sc  B.D. McKay} and {\sc A. Piperno}, Practical graph isomorphism, II, {\sl J. Symbolic Comput.} {\bf 60} (2014), 94--112.


\bibitem{MMN}
{\sc T. Miezaki}, {\sc A. Munemasa} and {\sc H. Nakasora}, A note on Assmus--Mattson type theorems, {\sl Des.  Codes Cryptogr.} {\bf 89}  (2021),  843--858.


\bibitem{MT}
{\sc A. Mohammadian} and  {\sc B. Tayfeh-Rezaie},  Hadamard matrices with few distinct types, {\sl Linear Multilinear Algebra}  {\bf 67} (2019),   1596--1605.


\bibitem{pal}
{\sc R.E.A.C. Paley},  On orthogonal matrices,  {\sl J. Math.  Phys.}  {\bf 12} (1933), 311--320.


\bibitem{RS-Handbook}
{\sc E.M.  Rains} and {\sc N.J.A. Sloane}, {\sl Self-dual codes}, in: Handbook of coding theory, Vol. I,  Elsevier,  Amsterdam, pp.  177--294, 1998.


\bibitem{REA}
{\sc R.C. Read}, Every one a winner or how to avoid isomorphism search when cataloguing combinatorial configurations, {\sl Ann.  Discrete Math.} {\bf  2}  (1978), 107--120.


\bibitem{wall}
{\sc J. Wallis},   Some $(1,-1)$ matrices, {\sl J. Combin.  Theory Ser. B} {\bf 10} (1971), 1--11.


\end{thebibliography}
\end{document}